\documentclass[12pt]{amsart}
   \usepackage{amsmath,amsthm}
   \usepackage{amsfonts}   
   \usepackage{amssymb}    
   \usepackage{url}
   \usepackage{pstricks}
      \usepackage{pstricks,graphicx}
\usepackage{comment}
\advance \topmargin by -\headheight
\advance \topmargin by -\headsep
      
\evensidemargin \oddsidemargin
\usepackage[all]{xy}

\newtheorem{thm}{Theorem}[section]
\newtheorem{prop}[thm]{Proposition}
\newtheorem{lem}[thm]{Lemma}

\newtheorem{defn}[thm]{Definition}

\theoremstyle{remark}

\title{Asymptotic properties of MMM-classes}
\author{Jonathan Bowden}
\address{Mathematisches Institut, Ludwig-Maximillians-Universit\"at, Theresienstr. 39, 80333 Munich, Germany}
\email{jonathan.bowden@math.lmu.de}
\date{\today}


\begin{document}

\maketitle

\begin{abstract}
We study geometric properties of characteristic classes of surfaces bundles. In particular, we show that oriented surface bundles over bases with amenable fundamental groups and dimension at least 2 have trivial simplicial volume. We show furthermore that all MMM-classes are hyperbolic in the sense of Gromov, verifying a weakened version of a conjecture due to Morita. Finally we consider surface bundles over products and restrictions on their characteristic classes
\end{abstract}

\section{Introduction}
In this article we discuss several properties of characteristic classes of surface bundles. Since the results of Madsen-Weiss \cite{MW} it is known that the set of stable rational characteristic classes of oriented surface bundles consists of so-called tautological or Mumford-Miller-Morita (MMM) classes $e_k$ and their products (cf.\ Definition \ref{def_MMM}). It therefore remains to investigate the properties of these classes more closely and we do this by looking at the asymptotic properties of these classes as well as their behaviour for certain classes of base manifolds.

It has been conjectured by Morita that all the MMM-classes are bounded in the sense of Gromov. It is known that the odd classes are bounded \cite{{Mor3}} so it remains to deal with the even classes. As a partial result in this direction Morita showed that all the rational MMM-classes vanish on \emph{amenable} subgroups (cf.\ \cite{Mor3}). In this article we will show that the MMM-classes are hyperbolic in the sense of Gromov (Theorem \ref{MMM_are_hyperbolic}), a concept that is strictly weaker than boundedness in general, but which implies for example vanishing on amenable groups. We also give an alternate proof of the fact that the rational MMM-classes vanish on amenable groups, which is based on reduction systems for subgroups of the mapping class groups and showing that the total space of a surface bundle over a base with amenable fundamental group has vanishing simplicial volume, unless the base is the circle:
\begin{thm}
Let $\Sigma \longrightarrow E \longrightarrow B$ be an oriented surface bundle over a closed, oriented manifold of dimension $dim(B) \geq 2$. If $\pi_1(B)$ is amenable, then the simplicial volume $||E||$ vanishes.
\end{thm}
Note that the assumption that the base has dimension at least $2$ is essential, since there are mapping tori over $S^1$ which are hyperbolic manifolds and thus have non-trivial simplicial volume.

In a similar direction using reduction systems we show how the MMM-classes can evaluate on non-trivial products:
\begin{thm}
Let $\Sigma \longrightarrow E \longrightarrow B$ be an oriented surface bundle over a base $B = M_1 \times M_2$ that is a non-trivial product. If $m = \text{max} \ \{ dim(M_1) , dim(M_2) \}$, then $e_k(E) = 0$ for all $k > \frac{m}{2}$.
\end{thm}
\noindent In particular, this theorem then says that the MMM-classes are indecomposable with respect to products.

\medskip

\noindent \textbf{Conventions:} All manifolds and bundles will be assumed to be oriented and smooth. 

\medskip

\noindent \textbf{Acknowledgement:} We thank Prof. D. Kotschick for helpful conversations and for introducing us to the notion of hyperbolicity. 

\section{Surface bundles and characteristic classes}\label{surface_bundles}
We begin by recalling certain generalities about surface bundles, which for the most part can be found in \cite{Mor}. Let $$\Gamma_h = Diff^+(\Sigma_h)/ Diff_0(\Sigma_h)$$ denote the mapping class group of an oriented Riemann surface $\Sigma_h$ of genus $h$. By the classical result of Earle and Eells \cite{EE} the identity component $Diff_0(\Sigma_h)$ is contractible in the $C^{\infty}$-topology if $ h \geq 2$. Thus the classifying space $B Diff^+(\Sigma_h)$ is homotopy equivalent to $B \Gamma_h$ which is in turn the Eilenberg-MacLane space $K(\Gamma_h,1)$.

In general, any oriented surface bundle is determined up to bundle isomorphism by the homotopy class of its classifying map and since $B Diff^+(\Sigma_h)$ is aspherical, a surface bundle $\Sigma_h \to E \to B$ is determined up to bundle isomorphism by the conjugacy class of its \emph{holonomy representation}:
\[ \rho: \pi_1(B) \longrightarrow \Gamma_h.\]
The conjugation ambiguity is a result of the choice of base points.
Conversely, any homomorphism $\rho: \pi_1(B) \longrightarrow \Gamma_h$ induces a map $$B \to K(\Gamma_h,1) = B \Gamma_h$$ and thus defines a bundle that has holonomy $\rho$.

In the presence of a marked point we can define the group of isotopy classes of diffeomorphisms fixing a marked point, which we denote by $\Gamma_{h,1}$. Furthermore there is a natural exact sequence given by forgetting the marked point
\[1 \longrightarrow \pi_1(\Sigma_h) \longrightarrow \Gamma_{h,1} \longrightarrow \Gamma_h \longrightarrow 1\]
and one may identify $B \Gamma_{h,1}$ with the the total space $E\Gamma_h$ of the universal bundle over $B \Gamma_h$.

There is a tautological cohomology class 
$$e \in H^2(B \Gamma_{h,1},\mathbb{Z}) = H^2(E \Gamma_h,\mathbb{Z})$$
which is defined as the Euler class of the oriented rank-2 bundle of vectors tangent to the fibers of $E \Gamma_h \longrightarrow B\Gamma_h$. Alternately, one can define $e$ as the Euler class associated to the central extension
\[1 \longrightarrow \mathbb{Z} \longrightarrow \Gamma^1_h \to \Gamma_{h,1} \to 1\]
where $\Gamma^1_h = Diff^c(\Sigma^1_h)/ Diff^c_0(\Sigma^1_h)$ denotes the mapping class group of a once punctured, genus $h$ surface. Here the right most map is given by collapsing the boundary to a point and the kernel is generated by a Dehn twist along a curve parallel to the boundary.

A family of characteristic classes of oriented surface bundles can be defined using the vertical Euler class. These are the so-called Mumford-Miller-Morita (MMM)-classes. 
\begin{defn}\label{def_MMM}
The (universal) $k$-th MMM-class of a surface bundle is defined as $$e_k= \pi_{!} e^{k+1} \in H^{2k}(B\Gamma_h,\mathbb{Z}),$$
where $e$ is the vertical Euler class of the universal oriented surface bundle $E\Gamma_h \longrightarrow B\Gamma_h$ and $\pi_{!}$ denotes integration along the fiber.

The $k$-th MMM-class of a particular bundle $E \longrightarrow B$ is denoted \\$e_k(E)~\in~H^{2k}(B,\mathbb{Z})$.
\end{defn}
If the genus of the fibre is at least $2$, then we can consider the $k$-th MMM-class as an element in the group cohomology of the mapping class group itself $e_k \in H^{2k}(\Gamma_h,\mathbb{Z})$.
\subsection{Boundedness of the vertical Euler class}
In \cite{Gro} Gromov introduced the so-called $l^1$-norm on homology that is defined as follows.
\begin{defn}[$l^1$-norm]
Let $ c = \sum_i \lambda_i \sigma_i$ be a chain in $C_k(X, \mathbb{R})$. We define the $l^1$-norm of $c$ to be
\[ || c ||_1 =  \sum_i |\lambda_i|.\]
For a class $\alpha \in H_k(X, \mathbb{R})$ we define
\[ ||\alpha||_1 = \inf \ \{ \  || z||_1 \ | \ \alpha = [z] \}. \]
\end{defn}
If $X$ is an orientable, closed manifold, then the norm of the fundamental class is called the \emph{simplicial volume} and is denoted $||X||$. By considering the natural pairing between homology and cohomology, one obtains a norm on cohomology that is dual to the $l^1$-norm.
\begin{defn}[$l^{\infty}$-norm]
Let $ c$ be a cochain in $C^k(X, \mathbb{R})$. We define the $l^{\infty}$-norm of $ c$ to be
\[ ||c||_{\infty} =  \sup_{\sigma \in S_k(X)}\frac{|c(\sigma)|}{\thinspace ||\sigma||_1}.\]
For a class $\alpha \in H^k(X, \mathbb{R})$ we define
\[ ||\alpha||_{\infty} = \inf \ \{ \  || c||_\infty \ | \ \alpha = [c] \}. \]
\end{defn}
This definition of the $\ell^{\infty}$-norm on cohomology agrees with the $\ell^{\infty}$-norm on cohomology as introduced by Gromov \cite{Gro}. We then say that a cohomology class is \emph{bounded}, if it is bounded with respect to the $l^{\infty}$-norm.

A fundamental fact first observed by Morita is that the vertical Euler class is bounded. This follows from the fact that the vertical Euler class can be described as a pull-back of the Euler class in the group cohomology of $\text{Homeo}_+(S^1)$, which is bounded. This fact can also be seen as a consequence of the adjunction inequality as explained in \cite{Bow}.
\begin{prop}[Morita, \cite{Mor3}]\label{Euler_bounded}
Let $h \geq 2$, then the vertical Euler class $e$ is bounded when considered as a class in $H^2(\Gamma_{h,1},\mathbb{R})$.
\end{prop}

\section{MMM-classes vanish on amenable groups}
In this section we show that all MMM-classes vanish on amenable groups as rational classes. 
Let us first recall the definition of amenability.
\begin{defn}
A group $G$ is called amenable if it admits a left-invariant mean, i.e.\ there is a left-invariant map from the set of bounded functions
$$\mu :  L^{\infty}(G) \to \mathbb{R}$$
such that $\mu(1)=1$ and $\mu(f) \geq 0$ if $f \geq 0$.
\end{defn}
The main tool we shall use to understand bundles over spaces with amenable fundamental groups is the theory of reduction systems for subgroups of mapping class groups (see \cite{BLM}, \cite{Iva2}). It will be convenient to use slightly different notation from the previous section. Namely, we consider a compact, orientable surface $\Sigma$ with possibly non-empty boundary and define
\[MCG(\Sigma) = PDiff^+(\Sigma) / Diff_0(\Sigma).\]
Here $PDiff^+(\Sigma)$ is the group of \emph{pure} orientation preserving diffeomorphisms of $\Sigma$, i.e.\ those diffeomorphisms that \emph{do not} permute boundary components, and $Diff_0(\Sigma)$ denotes those diffeomorphisms that are isotopic to the identity where the isotopy need \emph{not} fix the boundary. We will also want to consider the group of diffeomorphisms that fix the boundary up to isotopy which we denote by
\[MCG(\Sigma, \partial \Sigma) = PDiff^+(\Sigma, \partial \Sigma) / Diff_0(\Sigma, \partial \Sigma).\]
A subgroup $G \subset MCG(\Sigma)$ is called \emph{reducible} if there exists a homotopically non-trivial, embedded 1-dimensional submanifold $C \subset \Sigma$ which is \emph{componentwise} fixed by every element in $G$ up to isotopy. If $\Sigma$ has boundary we require that no component of $C$ is isotopic into the boundary. Such a submanifold is called a \emph{reduction system} for $G$. If no such $C$ exists, then we say that $G$ is \emph{irreducible}.

Next we consider a reducible subgroup $G \subset MCG(\Sigma)$. This then gives a map $G \longrightarrow MCG(\Sigma \setminus C)$. 
We let $MCG(\Sigma, C)$ denote the subgroup of the mapping class group that fixes each component of $C$ up to isotopy. If we let $\{Q_i\}$ for $1\le i \le k$ denote the closures of the components of $\Sigma \setminus C$, then the holonomy map factors through the map $MCG(\Sigma, C) \longrightarrow \prod_i MCG(Q_i)$. An important fact is that after taking a finite index subgroup $G' \subset G$, there is always a so-called \emph{maximal} reduction system $C_{max}$ so that the image of $G'$ in each $MCG(Q_i)$ is irreducible or trivial (\cite{Iva2}, Cor.\ 7.18).

Moreover, any \emph{irreducible} subgroup either contains a free group on two generators or is virtually cyclic (\cite{Iva2}, Cor.\ 8.6 and Theorem 8.9). Thus if $G$ is amenable it contains no free groups on two generators and hence the image $H_i$  of $G'$ in $MCG(Q_i)$ is virtually cyclic. After taking finite index subgroups one may then assume that each $H_i$ is either infinite cyclic, or trivial. The analogous result for solvable groups is older and goes back to Birman-Lubotzky-McCarthy in \cite{BLM}.

We summarise this discussion in the following theorem.
\begin{thm}[\cite{Iva2}, Theorem 8.9]\label{Iva_amenable}
Let $\Sigma$ be any compact surface and let $G \subset MCG(\Sigma)$ be amenable. Then $G$ is virtually abelian. Moreover, there exists a finite index subgroup $G' \subset G$ and a reduction system $C$ so that the images of $G'$ in $MCG(Q_i)$ are infinite cyclic or trivial.
\end{thm}

We are now able to state and prove the following theorem.
\begin{thm}\label{amenable_vanish}
Let $\Sigma \longrightarrow E \longrightarrow B$ be an oriented surface bundle over a closed, oriented manifold of dimension $dim(B) \geq 2$. If $\pi_1(B)$ is amenable, then the simplicial volume $||E||$ vanishes.
\end{thm}
\begin{proof}
Since the vanishing of the simplicial volume is unchanged under finite covers, we may assume by Theorem \ref{Iva_amenable} that the image of the holonomy map of $E$ is free abelian. That is $E$ is obtained as the pullback of some bundle $E'$ over a torus $T^N$ and we have the following commuting diagram:
\[\xymatrix{ E \ar[r]^{\bar{f}} \ar[d]_{\pi} & E' \ar[d]^--{\pi'} \\
B \ar[r]^{f} & T^N.}\]
Since the kernel of the induced map $\pi_1(E) \longrightarrow \pi_1(E')$ is amenable, Gromov's Mapping Theorem implies that $||[E]||_1 = ||\bar{f}_*[E]||_1$ (cf.\ \cite{Gro}, p.\ 40). Moreover, by applying the transfer homomorphism in homology to the above diagram we have
\[\xymatrix{ H_{n + 2}(E) \ar[r]^{\bar{f}_*}  & H_{n + 2}(E')  \\
H_n(B) \ar[r]^{f_*} \ar[u]^{\pi^{!}} & \ar[u]_{(\pi')^{!}} H_n(T^N).}\]
Since every class in $H_*(T^N)$ can be represented as a sum of tori, the class $f_{*}([B])$ can also be represented as a sum of tori. The commutativity of the above diagram then implies that the class $\bar{f}_*[E]$ is representable by a sum of the fundamental classes of several $\Sigma$-bundles over tori of dimension $n = dim(B)$. Thus it suffices to prove the theorem under the assumption that the base $B$ is a torus of dimension $n \geq 2$ and from now on we shall assume this.

We let $C_{max}$ be a maximal reduction system for the holonomy of $E$ which gives a fiberwise embedded $S^1$-bundle $\xi_i \subset E$ for each component of $C_{max}$. These $S^1$-bundles are $\pi_1$-injective and have amenable fundamental group. Thus Gromov's Cutting-off Theorem (cf.\ \cite{Gro}, p.\ 58) implies 
\[||E|| = ||E \setminus \bigcup_i \thinspace \xi_i||.\]
Since the holonomy group of each component $Q_i$ of $\Sigma \setminus C_{max}$ is either infinite cyclic or trivial, we see that each component of $E \setminus \bigcup_i \thinspace \xi_i$ is diffeomorphic to $M_i \times T^{n-1}$, where $M_i$ is a mapping torus with fiber $Int(Q_i)$ and holonomy $\psi_i$. This manifold admits proper self-maps of arbitrary degree, since $n \geq 2$. Hence the simplicial volume is either zero or infinite. As $E$ is closed we know that $||E|| < \infty$ and we conclude that
\[ ||E|| = ||E \setminus \bigcup_i \thinspace \xi_i|| = \sum_{i} ||M_i \times T^{n-1}|| = 0.  \]
\end{proof}
As a consequence of Theorem \ref{amenable_vanish} and the boundedness of the vertical Euler class (cf.\ Proposition \ref{Euler_bounded}) we will show that all MMM-classes vanish on amenable subgroups of $\Gamma_h = MCG(\Sigma_h)$.
\begin{thm}[Morita, \cite{Mor3}]\label{amenable_MMM}
The images of the MMM-classes in $H^*(G, \mathbb{Q})$ are trivial for amenable $G \subset \Gamma_h$.
\end{thm}
\begin{proof}
After taking a finite index subgroup we may assume that $G = \mathbb{Z}^N$ is free abelian. This subgroup corresponds to a surface bundle 
$$\Sigma_h \longrightarrow E \longrightarrow T^N$$ over the $N$-torus. Moreover, by Proposition \ref{Euler_bounded} the vertical Euler class is bounded with $||e||_{\infty} = C < \infty$ and thus $e^{k+1} \in H^{2k +2}(\Gamma_{h,1})$ is bounded with 
\[||e^{k+1}||_{\infty} \leq C^{k+1} .\]
The group $H_{2k}(T^N)$ has a basis consisting of (embedded) tori $T_j \hookrightarrow T^N$ and the theorem will follow if we show that $e_k$ is trivial on each $T_j$. 

We let $E_j$ denote the restriction of $E$ to $T_j$ and compute
\begin{align*}
|\langle e_k(E), [T_j] \rangle| & = |\langle \pi_! e^{k+1}, [T_j] \rangle|\\
& = |\langle e^{k+1}, \pi^![T_j] \rangle|\\
& = |\langle e^{k+1}, [E_j] \rangle| \\
& \leq ||e^{k+1}||_{\infty} \thinspace ||E_j|| \leq C^{k+1}||E_j|| .
\end{align*}
We may assume $\dim(T_j) \geq 2$, since $e^{k+1}$ is a cohomology class in degree at least four. Hence Theorem \ref{amenable_vanish} implies that $||E_j|| = 0$ and the result follows.
\end{proof}
\section{MMM-classes and products}
We will next investigate which classes in $H_*(\Gamma_h)$ can be represented as the image of the fundamental class of a non-trivial product $B = M_1 \times M_2$ of closed manifolds, where $dim(M_1),dim(M_2) > 0$. In particular, we will show that if $m = \text{max} \ \{ dim(M_1) , dim(M_2) \}$, then $e_k([B]) = 0$ for all $k > \frac{m}{2}$. This means that in particular $e_k$ vanishes modulo torsion for any bundle over a non-trivial product $N$ of dimension $dim(N) = 2k$. We begin by proving the following lemma.
\begin{lem}\label{MMM_prod}
Let $\Sigma$ be a closed, connected, oriented surface and let $C$ be a disjoint collection of embedded circles on $\Sigma$. We let $Q_j$ be the components of $\Sigma \setminus C$ and let $\overline{Q}_j$ be the closed surface obtained from $Q_j$ by identifying each boundary component to a point. We further let $\bar{\rho}_j$ be the natural map $MCG(\Sigma, C) \to MCG(\overline{Q}_j)$. Then the $k$-th MMM-class on $MCG(\Sigma, C)$ satisfies
\[e_k = \sum_{j = 1}^n \bar{\rho}_j^* e_k.\]
\end{lem} 
\begin{proof}
For simplicity let $G_j = MCG(Q_j)$ and $G_C = MCG(\Sigma,C)$. We also let $\overline{Q}_j$ denote the closed surface obtained by identifying each boundary component of $Q_j$ to a point and $\overline{G}_j = MCG(\overline{Q}_j)$ the corresponding mapping class group.

By definition, the universal bundle over $BG_C$ has a natural decomposition
\[E = \bigcup_{j = 1}^n E_j,\]
where $E_j$ is a bundle with fiber $Q_j$. Moreover, the vertical bundle is trivial over $\partial E_j$ with a trivialisation given by taking vectors tangent to the boundary. We let $\xi_C$ denote the union of the $S^1$-bundles corresponding to $\partial E_j$. Then the vertical vector bundle on $E$ descends to a bundle on the quotient space $E^* = E / \xi_C$. Similarly the vertical bundle on $E_j$ descends to $E^*_j = E_j / \partial E_j$ and we note that $E^* = \bigvee_{j = 1}^n E_j^*$. Since $E^*$ is a wedge sum and the Euler class is natural under pullbacks, we compute 
\begin{align}\label{sum_formula}
e^{k+1}(E) = \sum_{j = 1}^n e^{k+1}(E^*_j).
\end{align}
We let $\overline{E}_j$ denote the bundle obtained from $E_j$ by fiberwise identifying each boundary component of $Q_j$ to a point and we let $BG_C \stackrel{\bar{\rho}_j} \longrightarrow B \overline{G}_j$ be the classying map of this bundle. We also let $\overline{E}$ denote the union of the $\overline{E}_j$. In this way we obtain the following commuting diagram:
\[\xymatrix{ & E_j^* & & \\
E \ar[d]^{\pi} \ar[ur] \ar[r] & \overline{E} \ar[d] \ar[u] & \overline{E}_j \ar[d]^{\overline{\pi}_j} \ar[l] \ar[r] \ar[ul] & E \overline{G}_j \ar[d]^{\bar{p}_j}\\
 B G_C \ar[r]^{Id} & B G_C  & B G_C \ar[r]^{\bar{\rho}_j} \ar[l]_{Id}& B\overline{G}_j.}\]
We conclude that $\pi_! (e^{k+1}(E^*_j)) = \bar{\rho}_j^* e_k(E \overline{G}_j)$ and the lemma follows by equation (\ref{sum_formula}).
\end{proof}
We may now prove the following theorem.
\begin{thm}\label{product_MMM}
Let $\Sigma \longrightarrow E \longrightarrow B$ be an oriented surface bundle over a base $B = M_1 \times M_2$ that is a non-trivial product. If the dimension $m = \text{max} \ \{ dim(M_1) , dim(M_2) \}$, then $e_k(E) = 0$ as a rational class for all $k > \frac{m}{2}$.
\end{thm}
\begin{proof}
We let $G_i = \pi_1(M_i)$ and we also let $G = G_1 \times G_2 \stackrel{\rho} \longrightarrow MCG(\Sigma)$ denote the holonomy map of $E$. If the image of $G$ lies in the kernel of the modular map
\[MCG(\Sigma) \stackrel{\Phi_3}  \longrightarrow Aut(H_1(\Sigma, \mathbb{Z}_3))\]
then the existence of a maximal reduction system is guaranteed by (\cite{Iva2}, Cor. 7.18). 
Thus after taking finite index subgroups of the $G_i$ we may assume that this is the case and without loss of generality we have a maximal reduction system $C_{max}$ so that the images in $G \stackrel{\rho_i} \longrightarrow MCG(Q_i)$ are trivial or irreducible. If $\rho_i(G)$ is non-trivial it must contain a pseudo-Anosov element $\phi = \rho_i(a,b)$. Since the subgroup generated by $\alpha = \rho_i(a,e)$ and $\beta = \rho_i(e , b)$ is irreducible, abelian and consists of elements in the kernel of the modular map $\Phi_3$, it must be infinite cyclic and is generated by a pseudo-Anosov element $\psi$ (\cite{Iva2}, Cor. 7.14 and Cor. 8.6). In particular, $\alpha$ is pseudo-Anosov if it is non-trivial and its centraliser $C(\alpha)$ in the kernel of $\Phi_3$ is infinite cyclic (\cite{Iva2}, Lem. 8.13).

Without loss of generality we assume that the $\alpha$ defined above is non-trivial. Then since the subgroup $\rho_i(G_2)$ commutes with $\alpha$ it is contained in the centralizer of $\alpha$ and it follows that $\rho_i(G_2)$ is infinite cyclic and generated by a pseudo-Anosov element or it is trivial. If it is non-trivial then the fact that $G_1$ commutes with $G_2$ again implies that the image $\rho_i(G)$ is also cyclic. Thus we conclude that either $\rho_i$ factors through one of the projections $G \stackrel{\pi_i} \longrightarrow G_i$ or that the image is cyclic (or trivial).

We let $\Sigma_1$ denote the subsurface of $\Sigma$ on which $\rho(G_1)$ is non-trivial but $\rho(G_2)$ is trivial. Similarly, we let $\Sigma_2$ be the subsurface where $\rho(G_2)$ is non-trivial but $\rho(G_1)$ is trivial. We finally let $\Sigma_3$ be the subsurface on which the holonomy is component-wise cyclic or trivial. We let $\bar{\rho}_j$ denote the induced maps to $MCG(\overline{\Sigma}_j)$ and by applying Lemma \ref{MMM_prod} we conclude that
\[e_k = \sum_{j = 1}^3 \bar{\rho}_j^* e_k.\]
The first two summands vanish for dimension reasons if $k > \frac{m}{2}$ and since the image of $\bar{\rho}_3$ is abelian the third vanishes by Theorem \ref{amenable_MMM}. 
\end{proof}
We contrast the above result with those of Morita in \cite{Mor}. In particular, Morita showed that for sufficiently large genus any MMM-class is detected by an iterated surface bundle given by what is now called the Morita $m$-construction. Repeated application of Theorem \ref{product_MMM} implies that the only MMM-classes that are possibly non-trivial over a base that is a product of surfaces are of the form $e_1^k$. Moreover, it can be shown that these classes can be detected by products of Riemann surfaces if the genus of the fiber satisfies $h \geq 3k$. In fact, the proof of Theorem \ref{product_MMM} means that this bound is sharp for such bundles, since $e_1$ is trivial for bundles with fiber of genus $h \leq 2$.

\section{MMM-classes are hyperbolic}\label{MMM_Gromov_hyp}
As previously mentioned Morita has conjectured that the MMM-classes have representatives that are bounded in the sense of Gromov. In particular, he showed that $e_k$ vanishes on amenable groups (cf.\ Theorem \ref{amenable_MMM}). In studying certain properties of K\"ahler manifolds Gromov introduced the weaker notion of \emph{hyperbolicity}, which is called $\tilde{d}(\textrm{boundedness})$ in \cite{Gro2}. We will extend Morita's original argument to show that the MMM-classes are hyperbolic.

We shall first recall the definition of hyperbolicity for simplicial complexes following \cite{BruKot}. To this end we need to consider metrics and differential forms on simplicial complexes. Recall that a metric on a simplicial complex is given by a metric $g_{\sigma}$ on each simplex $\sigma$ so that when $\tau \subset \sigma$ is a face, one has $g_{\sigma}|_{\tau} = g_{\tau}$. Similarly, a differential form is a collection of forms on each simplex compatible with restriction to faces. One can then define the exterior derivative on each simplex and the resulting cohomology is isomorphic to ordinary cohomology for simplicial complexes (cf.\ \cite{Swan}).
\begin{thm}[Simplicial de Rham Theorem]
Let $X$ be a simplicial complex. Then there is a natural isomorphism $H^k_{dR}(X) \stackrel{\Psi} \longrightarrow H_{\Delta}^k(X, \mathbb{R})$ from de Rham cohomology to simplicial cohomology given by integration over chains.
\end{thm}
The de Rham isomorphism $\Psi$ has a natural inverse on the chain level. This is defined as follows: let $\sigma$ be an oriented simplex of $X$ and let $\mu_i$ denote the baracentric coordinate map defined by the the $i$-th vertex $v_i$ of $\sigma$. That is for a simplex $\tau$ of $X$ we define $\mu_i |_{\tau}$ to be zero if $v_i$ is not a vertex of $\tau$, otherwise we let $\mu_i |_{\tau}(p)$ be the coefficient of $v_i$ given by writing $p$ as a convex combination of the vertices of $\tau$. The functions $\mu_i$ are well-defined elements in $\Omega^0(X)$. Then for any oriented cosimplex $\sigma^* \in C_{\Delta}^k(X, \mathbb{R})$ we define
\[\Phi_{\sigma^*} = k ! \sum_{i^=0}^k (-1)^{i} \mu_i d\mu_0 \wedge ... \wedge \widehat{d\mu_i} \wedge... d \mu_k.\]
This is a so-called elementary $k$-form and has support in the set $st(\sigma)$. For an arbitrary simplicial cochain $c = \sum \lambda_{\sigma} \sigma^*$ we set
\[\Phi(c) =  \sum \lambda_{\sigma} \Phi_{\sigma^*}.\]
and this map is the desired inverse of $\Phi$ (cf.\ \cite{Whit}, p.\ 229 ff).

Recall that a $k$-form $\alpha \in \Omega^k(M)$ on a manifold is bounded if
\[|| \alpha ||_g = \sup_{x \in M} | \alpha_x (e_1,...e_k)| < \infty\]
for all $k$-tuples of orthonormal vectors in $T_x M$. For a simplicial complex a form is bounded if there is a universal bound over all simplices. 

For a simplicial cochain $c \in C_{\Delta}^k(X, \mathbb{R})$ one also has a notion of boundedness. Indeed, one has the $L^{\infty}$-norm 
\[||c||_{\Delta} = \sup_{\sigma \in S_{\Delta}^k(X)} |c(\sigma)|\]
where the supremum is taken over all $k$-simplices $\sigma$ of $X$. If this number is finite then $c$ is said to be bounded. Moreover, the set of bounded simplicial chains is a subcomplex of $C^*_{\Delta}(X, \mathbb{R})$ that we denote by $\hat{C}^*_{\Delta}(X, \mathbb{R})$. Under certain fairly natural assumptions the de Rham isomorphism and its inverse preserve boundedness:
\begin{prop}\label{bound_deRham}
Let $X$ be a simplicial complex and let $g$ be a metric so that for all $k$-simplices $Vol(\sigma, g) \leq C_k$. Then the map \linebreak $\Omega^k(X) \stackrel{\Phi} \longrightarrow C_{\Delta}^k(X, \mathbb{R})$ given by integration over chains preserves boundedness.

Conversely, assume that the star of each simplex of $X$  contains a bounded number of $k$-simplices for some universal constant $S_k$ and that $g$ is a metric on $X$ so that the $1$-forms $d\mu_i$ given by baracentric coordinates are uniformly bounded. Then the inverse of the de Rham isomorphism preserves boundedness.

In particular, if $X \stackrel{p} \longrightarrow Y$ is a (simplicial) covering map and $Y$ is finite, then $X$ endowed with the pullback metric satisfies the hypotheses above.
\end{prop}
\begin{proof}
We let $\Psi$ denote the de Rham isomorphism, given by integration over chains. For the first statement note that for any $k$-form $\omega$ and any $k$-simplex 
\[ |\Psi(\omega)(\sigma)| = \left|\int_{\sigma} \omega\right| \leq Vol(\sigma, g) ||\omega||_g \leq C_k ||\omega||_g\]
and hence $\Psi(\omega)$ is a bounded simplicial cochain.

Conversely, let $c$ be a bounded simplicial cochain that we write as $c = \sum \lambda_{\sigma} \sigma^*$. Then
\[\Phi(c) =  \sum \lambda_{\sigma} \Phi_{\sigma^*}\]
and the $\lambda_{\sigma}$ are bounded by definition. Moreover, the $\Phi_{\sigma^*}$ are elementary $k$-forms and as the forms $d\mu_i$ are uniformly bounded the same holds for $\Phi_{\sigma^*}$. Note that every point $p \in X$ lies in the interior of a unique simplex $\tau_p$ and $\Phi_{\sigma^*}(p)$ is necessarily zero unless $\sigma$ lies in $st(\tau_p)$. Thus $\Phi_{\sigma^*}(p)$ is non-zero for at most $S_k$ simplices and it follows that $\Phi(c)$ is uniformly bounded.

Finally, if $g_X = p^*g$ is the pullback metric on $X$, then the volumes of simplices are the same as their images in $Y$ and these are uniformly bounded by the assumption that $Y$ is finite. The same holds for the 1-forms $d\mu_i$, since these are locally pullbacks of the corresponding forms on $Y$. Moreover, the star of each simplex in $X$ has at most $S_k$ simplices, where $S_k$ denotes the number $k$-simplices in $Y$ which is finite by assumption, thus proving the final claim.
\end{proof}
We now consider a finite simplicial complex $X$ with a metric $g$ and let $\widetilde{g}$ denote the pullback metric on the universal cover $\widetilde{X} \stackrel{p} \longrightarrow X$. With this notation we have the following definition of hyperbolicity of cohomology classes on finite complexes.

\begin{defn}[Hyperbolicity for cohomology]
A class \linebreak $\alpha \in H^k(X, \mathbb{R})$ is called \emph{hyperbolic} if there exists a de Rham representative $\eta \in \Omega^k(X)$ of $\alpha$, so that the $p^* \eta$ has a bounded primitive with respect to the metric $\widetilde{g}$. 
\end{defn}

By Proposition \ref{bound_deRham} this is equivalent to the statement that the simplicial cochain $p^* \Psi(\eta) = \Psi(p^* \eta) \in \hat{C}^*_{\Delta}(\widetilde{X}, \mathbb{R})$ is exact as a bounded simplicial cochain. Hence it is clear that the definition is independent of the metric and the chosen representative $\eta$. Furthermore, since any continuous map can be approximated by a simplicial map, hyperbolicity is natural under maps between finite complexes. 

More generally, if $Y$ is any topological space, then we make the following definition.
\begin{defn}
Let $Y$ be a topological space. A class $\alpha \in H^k(Y,\mathbb{R})$ is hyperbolic if $f^* \alpha$ is hyperbolic for every continuous map $X \stackrel{f} \longrightarrow Y$ of a finite complex $X$ to $Y$.
\end{defn}
As in the case of bounded cohomology, all hyperbolic classes are trivial if $\pi_1(X)$ is amenable. The proof of Brunnbauer and Kotschick in \cite{BruKot} is geometric and uses certain isoperimetric inequalities. One can however give a more direct proof that follows Gromov's original argument in the bounded case:
\begin{thm}\label{amenable_triv}
Let $X$ be a finite simplicial complex with amenable fundamental group. Then all hyperbolic classes are trivial.
\end{thm}
\begin{proof}
We let $\widetilde{X} \stackrel{p} \longrightarrow X$ denote the universal cover and let $G = \pi_1(X)$. Then $G$ acts on $\widetilde{X}$ by (simplicial) deck transformations that we denote by $T_g$. Now assume that $\alpha \in H^k(X, \mathbb{R})$ is a hyperbolic class. By Proposition \ref{bound_deRham} this means that for any simplicial representative $a$ of $\alpha$ the cochain $p^* a \in \hat{C}^*_{\Delta}(\widetilde{X}, \mathbb{R})$ is exact. We let $b \in \hat{C}^{k - 1}_{\Delta}(\widetilde{X}, \mathbb{R})$ be a primitive. Then since $G$ is amenable there is an averaging operator on bounded \emph{singular} chains
\[ C_b^k(\widetilde{X}, \mathbb{R}) \stackrel{A} \longrightarrow C_G^k(\widetilde{X}, \mathbb{R})\]
that maps an arbitrary bounded cochain to a $G$-equivariant one. This map is defined as follows: let $\mu: L^{\infty}(G) \to \mathbb{R}$ be a left-invariant mean, which exists since $G$ is amenable. Let $c \in C_b^k(\widetilde{X}, \mathbb{R})$ and let $\sigma$ be any $k$-simplex, we define a function $\phi_{c, \sigma}: G \to \mathbb{R}$ by
\[\phi_{c, \sigma}(g) = c((T_{g^{-1}})_* \thinspace \sigma).\]
We then set 
\[A(c)(\sigma) = \mu(\phi_{c, \sigma}).\]
Since $\mu$ was left invariant $A(c)$ is a $G$-equivariant cochain on $\widetilde{X}$, that is $A(c) = p^* c'$ for a unique cochain in $C^k(X, \mathbb{R})$. One also checks that $A$ is a chain map. Finally, as the deck transformations are simplicial $A$ induces a well-defined map on bounded simplicial cochains. If we let $b' \in C_{\Delta}(X, \mathbb{R})$ be such that $A(b) = p^* b'$ we compute
\[p^* \delta b' =  \delta p^* b'= \delta A(b) = A(p^* a) = p^* a.\]
Thus $\delta b' = a$ since $p^*$ is injective and the class $\alpha \in H^k(X, \mathbb{R})$ is trivial.
\end{proof}
In order to show that the MMM-classes are hyperbolic, we shall need two technical lemmata, the first of which is in essence Theorem 2.1 in \cite{Ked}. In \cite{Ked} K\c{e}dra considered only the universal cover of a manifold, however our assumption that $p^* \alpha$ is exact in bounded cohomology ensures that his proof goes through.
\begin{lem}\label{Kedra_bounded}
Let $\bar{X} \stackrel{p} \longrightarrow X$ be a covering of simplicial complexes, with $X$ finite. Let $\alpha \in H_b^k(X, \mathbb{R})$ be a bounded cohomology class such that $p^* \alpha$ is trivial in $H_b^k(\bar{X}, \mathbb{R})$. Then there is a de Rham representative $\Phi_{\alpha}$ of $\alpha$ and a bounded $(k-1)$-form $\Phi_{\beta}$ with 
\[d \Phi_{\beta} = p^* \Phi_{\alpha}.\]
\end{lem}
\begin{proof}
We first lift the simplices of $X$ to $\bar{X}$ and let $\bar{g}$ be the lifted metric. Now let $\beta$ be a bounded singular $(k-1)$-cochain on $\bar{X}$ so that $\delta \beta = p^* \alpha$. By restricting to the simplicial cochain complex, we obtain a simplicial cochain
\[\beta_s = \sum \lambda_{\bar{\sigma}} \bar{\sigma}\]
where the $\lambda_{\bar{\sigma}}$ are bounded and $\delta \beta_s = p^*\alpha_s$ as simplicial cochains. Applying the inverse of the de Rham isomorphism to $\alpha_s, \beta_s$ we obtain forms $\Phi_{\alpha}, \Phi_{\beta}$ such that $d = p^* \Phi_{\alpha}$ and by Proposition \ref{bound_deRham} the form $\Phi_{\beta} $ is bounded.
\end{proof}
The next lemma gives sufficient conditions under which integration along the fiber maps bounded forms to bounded forms.
\begin{lem}\label{bounded_int}
Let $F \longrightarrow E \stackrel{\pi} \longrightarrow M$ be a smooth fiber bundle over a manifold $M$, whose fiber is a closed manifold of dimension $m$. Let $g_M$ be a metric on $M$ and $g_E$ a submersion metric on $E$. Let $\Omega_v$ denote the fiberwise volume form induced by $g_E$. If $\pi_{!}\Omega_v$ is bounded, then the map
\[ \pi_{!} : \Omega^{k + m}(E) \to  \Omega^{k}(M)\]
maps bounded forms to bounded forms.
\end{lem}
\begin{proof}
Let $\phi$ be a bounded $(k +m)$-form with respect to the metric $g_E$. Let $U$ be an open neighbourhood of $p \in B$ and choose a trivialisation $V = \pi^{-1}(U) \cong U \times F$. We let $\{ e_1,...,e_n \}$ be a local orthonormal frame on $U$ with respect to $g_B$ and $\{ e^1,..,e^n \}$ its dual. On $V$ we may decompose $\phi$ as
\[ \phi = \sum_I (f_I \Omega_v) \wedge (\pi^* e^I) + \psi\]
where $I$ is a multi-index of length $k$ and $\psi$ vanishes on $k$-tuples of vertical vectors. Choose a local orthonormal frame  $f_1, ... f_k$ about a point $(p,x)$ in $V$ and let $\bar{e}_1,...\bar{e}_n$ be a lift of this frame to $E$ that is guaranteed by the assumption that $g_E$ is a submersion metric. Then $\{f_1,...f_k, \bar{e}_1,...\bar{e}_n \}$ is an orthonormal frame and hence as $\phi$ is bounded
\[|\phi( f_1,...f_k, \bar{e}_1,...\bar{e}_n)(p,x)| = |f_I(p,x)| < C.\]
Thus we conclude locally
\[\pi_{!} \phi(p) = \sum_I \int_{F_p} (f_I\Omega_v) e^I\]
and
\[\left|\int_{F_p} f_I \Omega_v\right| \leq C \left|\pi_{!} \Omega_v(p)\right|\]
so $\pi_{!} \phi$ is bounded if $\pi_{!} \Omega_v$ is.
\end{proof}
With the aid of these results we will prove the hyperbolicity of the MMM-classes.
\begin{thm}\label{MMM_are_hyperbolic}
The MMM-classes are hyperbolic.
\end{thm}
\begin{proof}
We let $e$ denote the vertical Euler class and we also let $X \stackrel{f} \longrightarrow B \Gamma_h$ be the classifying map of a bundle $E$ over a finite simplicial complex $X$. We first assume that $X = B$ is a smooth, compact manifold (possibly with boundary). We let $\widetilde{E} $ be the pullback bundle over the universal cover of $B$
\[\xymatrix{\widetilde{E} \ar[r]^{p} \ar[d]_{\pi} & E \ar[d]^{\pi} \ar[r] & E \Gamma_h \ar[d]\\
\widetilde{B} \ar[r]^{p} & B \ar[r]^f & B \Gamma_h }\]
and further let $\Sigma_h \stackrel{\iota} \longrightarrow \widetilde{E}$ be the inclusion of a fiber. Morita has shown that $\iota^* e^2$ is trivial in \emph{bounded cohomology} (cf.\ \cite{Mor3}, Section 6). Thus the same holds for $\iota^* e^{k +1}$. Moreover, since $\pi_1(\widetilde{E}) \cong \pi_1(\Sigma_g)$ this inclusion induces an isomorphism on bounded cohomology. We may thus choose a bounded chain $b_k \in C^{2k -1}_b(\widetilde{E}, \mathbb{R})$ with $p^* e^{k +1} = \delta b_k$. 

Then by Lemma \ref{Kedra_bounded} there is a form $\Phi_k$ on $\widetilde{E}$ which is bounded with respect to the pullback metric and a form $\Psi_{k+1}$ on $E$ that is a representative of $e^{k+1}$ so that $p^* \Psi_{k +1} = d \Phi_k$. Since integration along the fiber is natural and commutes with the exterior derivative, we compute
\begin{align*} 
p^* e_k & = p^* \pi_{!} \Psi_{k+1} = \pi_{!} p^* \Psi_{k+1}\\
& = \pi_{!} d \Phi_k = d (\pi_{!} \Phi_k)
\end{align*}
We finally need to check that $\pi_{!} \Phi_k$ is bounded with respect to the pullback metric on $\widetilde{B}$. Let $g_B$ be any metric on the base and let $g_E$ be a submersion metric on $E$. The pullback metric $\widetilde{g} = p^* g_E$ is a submersion metric for $g_{\widetilde{B}} = p^* g_B$ and the vertical volume form $\widetilde{\Omega}_v$ on $\widetilde{E}$ is the pullback of the vertical volume form $\Omega_v$ on $E$ and thus
\[\pi_{!} \widetilde{\Omega}_v = \pi_{!} p ^*  \Omega_v = p ^*  \pi_{!} \Omega_v\]
is a pullback of a form on $B$ and is hence bounded. Now $\phi_k$ is bounded with respect to the metric $\widetilde{g}$ and thus by Lemma \ref{bounded_int} it follows that $\pi_{!} \Phi_k$ is bounded. Hence $e_k \in H^{2k}(B)$ is hyperbolic.

In the general case let $X$ be an arbitrary finite simplicial complex. We may embed $X$ in $\mathbb{R}^N$ for some sufficiently large $N$. We then let $B = \nu(X)$ be a (compact) regular neighbourhood of $X$ in $\mathbb{R}^N$. Since $\nu(X)$ deformation retracts onto $X$ we have the following commutative diagram
\[\xymatrix{ X \ar[r] \ar[dr]_f & \nu(X) \ar[d]^{\bar{f}} \\
& B \Gamma_h.}\]
Then by the argument above, $\bar{f}^* e_k$ is hyperbolic and hence by naturality so is $f^* e_k$.
\end{proof}
In general, the notion of hyperbolicity is strictly weaker than that of boundedness - the reason being that hyperbolicity is preserved under cup products so that the space of hyperbolic cohomology classes forms an ideal, whereas the space of bounded classes is only a subring. In particular the dual of the fundamental class of any product $T^n \times M$ is hyperbolic if $||M|| \neq 0$, but $[T^k \times M]^* \in H^{k+n}(T^k \times M) $ is not bounded. However, in the case of classes of degree two it is still open as to whether hyperbolicity implies boundedness. In terms of group cohomology the hyperbolicity condition seems to be closely related, or even equivalent to the notion of weak boundedness considered in \cite{NR}. This might provide a better basis for finding examples of hyperbolic classes of degree $2$ that are not bounded.


\begin{thebibliography}{99}
  
\bibitem{BLM} J. Birman, A. Lubotzky and J. McCarthy,
\textit{Abelian and Solvable Subgroups of Mapping Class Groups},
Duke Math. J. \textbf{50} (1983),  no. 4, 1107--1120.


\bibitem{Bow} J. Bowden,
\emph{Closed leaves of foliations, multisections and stable commutator lengths},
J. Topol. Anal. \textbf{4} (2011), no. 3, 491--509.

\bibitem{BruKot} M. Brunnbauer and D. Kotschick,
\emph{On hyperbolic cohomology classes},
preprint 2009, arXiv:0808.1482v1.

\bibitem{EE} C. J. Earle and J. Eells,
\emph{The diffeomorphism group of a compact Riemann surface},
Bull. Amer. Math. Soc. \textbf{73} (1967), 557--559.


\bibitem{Gro} M. Gromov,
\emph{Volume and bounded cohomology},
Inst. Hautes \'Etudes Sci. Publ. Math. \textbf{56} (1982), 5--99.

\bibitem{Gro2} M. Gromov, 
\emph{K\"ahler hyperbolicity and $L^2$-Hodge theory},
J. Differential Geometry \textbf{33} (1991), 263--292.

\bibitem{Iva} N. Ivanov,
Mapping class groups. \emph{Handbook of geometric topology}, 523--633, North-Holland, Amsterdam, 2002.
\bibitem{Iva2} N. Ivanov,
\emph{Subgroups of Teichm\"uller Modular Groups},
Translations of Mathematical Monographs \textbf{115}, AMS, 1992.

\bibitem{Ked} J. K\c{e}dra,
\emph{Symplectically hyperbolic manifolds},
Differential Topol. Appl. \textbf{27} (2009), no. 4, 455--463.

\bibitem{KL} D. Kotschick and C. L\"{o}h,
\emph{Fundamental classes not representable by products},
J. Lond. Math. Soc. (2) \textbf{79} (2009), no. 3, 545--561.

\bibitem{MW} I. Madsen and M. Weiss,
 \emph{The stable moduli space of Riemann surfaces: Mumford's conjecture},
 Ann. of Math. (2) \textbf{165} (2007), no. 3, 843--941.

\bibitem{Mor} S. Morita,
\emph{Characteristic classes of surface bundles},
Invent. Math. \textbf{90} (1987), 551--577.

\bibitem{Mor3} S. Morita, \emph{Characteristic classes of surface bundles and bounded cohomology},
A F\^{e}te of Topology, Academic Press, 1988, 233--258.

\bibitem{NR} W. Neumann and L. Reeves,
\emph{Regular cocycles and biautomatic structures},
Internat. J. Algebra Comput. \textbf{6} (1996), no. 3, 313--324. 

\bibitem{Swan} R. G. Swan,
\emph{Thom's theory of differential forms on simplicial sets},
Topology \textbf{14} (1975), 271--273.

\bibitem{Whit} H. Whitney,
\emph{Geometric Integration Theory},
Princeton University Press, 1957.



\end{thebibliography}
\end{document}